\documentclass[reqno]{amsart}
\newcommand{\ben}{\begin{enumerate}}
\newcommand{\een}{\end{enumerate}}
\newcommand{\bqu}{\begin{quote}}
\newcommand{\equ}{\end{quote}}
\newcommand{\beq}{\begin{equation}}
\newcommand{\eeq}{\end{equation}}
\newcommand{\bec}{\begin{center}}
\newcommand{\ece}{\end{center}}

\usepackage{url}
\usepackage{latexsym,epsfig,amssymb,amsmath,amsthm,color,url}
\usepackage[shortlabels]{enumitem}
\usepackage{hyperref}

 \numberwithin{equation}{section}
\newtheorem{theorem}{Theorem}[section]
\newtheorem{prop}[theorem]{Proposition}

\newtheorem{lemma}[theorem]{Lemma}

\DeclareMathOperator{\unif}{Uniform}

\theoremstyle{definition}
\newtheorem{defn}[theorem]{Definition}

\theoremstyle{remark}

\title[Percolative entropy over infinite regular trees]{Gibbs measures over locally tree-like graphs and percolative entropy over infinite regular trees}
\date{}
\author{Tim Austin}
\address{Tim Austin, \ Courant Institute of Mathematical Sciences, \ New York University, \ 251 Mercer Street, New York, NY 10012, United States.}
\email{tim@cims.nyu.edu.}
\author{Moumanti Podder}
\address{Moumanti Podder, \ Courant Institute of Mathematical Sciences, \ New York University, \ 251 Mercer Street, New York, NY 10012, United States.}
\email{mp3460@nyu.edu.}

\begin{document}
\bibliographystyle{plain}

\begin{abstract}
Consider a statistical physical model on the $d$-regular infinite tree $T_{d}$ described by a set of interactions $\Phi$. Let $\{G_{n}\}$ be a sequence of finite graphs with vertex sets $V_n$ that locally converge to $T_{d}$.  From $\Phi$ one can construct a sequence of corresponding models on the graphs $G_n$.  Let $\{\mu_n\}$ be the resulting Gibbs measures. Here we assume that $\{\mu_{n}\}$ converges to some limiting Gibbs measure $\mu$ on $T_{d}$ in the local weak$^*$ sense, and study the consequences of this convergence for the specific entropies $|V_n|^{-1}H(\mu_n)$. We show that the limit supremum of $|V_n|^{-1}H(\mu_n)$ is bounded above by the \emph{percolative entropy} $H_{perc}(\mu)$, a function of $\mu$ itself, and that $|V_n|^{-1}H(\mu_n)$ actually converges to $H_{perc}(\mu)$ in case $\Phi$ exhibits strong spatial mixing on $T_d$. We discuss a few examples of well-known models for which the latter result holds in the high temperature regime.
\end{abstract}

\subjclass[2010]{}


\maketitle

\section{Introduction}

Many basic models of statistical physics are obtained by specifying an energy function on a configuration space of the form $A^V$, where $A$ is a finite set of possible `single-particle states' and $V$ is a set of particles.  Given the energy function, one asks after its partition function or free energy, or forms the associated Gibbs measure on $A^V$ and attempts to compute or estimate its marginals.

Many classical models, such as the Ising and Potts models, are constructed using a graph $G$ on $V$.  Such models can be defined consistently for many different graphs, and many important methods are based on comparing their behaviours over different graphs.

Let $d\geq 1$ be a fixed integer.  In this paper we study a class of models that can be defined consistently over either a finite $d$-regular graph or over the infinite $d$-regular tree.  In the latter setting the associated Gibbs states are defined by the Dobrushin--Lanford--Ruelle conditions, and there may be several of them.

The main focus of this paper is the connection between the Shannon entropies of the Gibbs measures on each of a sequence of large finite graphs, where the sequence converges locally to an infinite regular tree, and a notion of entropy that applies to a translation-invariant measure over the tree itself: percolative entropy, defined below.

We study models in Ruelle's thermodynamic formalism, and largely follow~\cite[Chapter 1]{ruelle} for notation and terminology. Another reference which includes the same background is~\cite{georgii}.

\par Suppose first that the set of particles $V$ is also finite. Let $\mathcal{F}$ be a family of subsets of $V$. An \emph{interaction} is a function $\Phi: \bigcup_{S \in \mathcal{F}} A^S \rightarrow [-\infty, \infty)$. The Gibbs measure $\mu^{\Phi}$ with respect to $\Phi$ is defined as
\begin{equation} \label{Gibbs}
\mu^{\Phi}(\xi) = \frac{1}{Z^{\Phi}} \exp\left\{\sum_{S \in \mathcal{F}} \Phi(\xi_{S})\right\}, \ \xi \in A^{V},
\end{equation}
where $\xi_S$ is the projection of $\xi$ to coordinates in $S$, and $Z^{\Phi}$ is the partition function defined as
\begin{equation} \label{partition function infinite}
Z^{\Phi} = \sum_{\xi \in A^{V}} \exp\left\{\sum_{S \in \mathcal{F}} \Phi(\xi_{S})\right\}.
\end{equation}
This Gibbs measure is well defined if at least one $\xi \in A^{V}$ has finite energy, i.e. $\sum_{S \in \mathcal{F}} \Phi(\xi_{S}) > -\infty$, and in this case $\mu$ is supported only on those configurations $\xi$ that have this property. We assume that at least one such $\xi$ exists, in the rest of the paper.

\par If $V$ is infinite, this direct definition is not available, since the sums inside the exponentials above are generally also infinite.  In that case we define associated Gibbs states as follows. Suppose now that $\mathcal{F}$ is a family of finite subsets of $V$ which is \emph{locally finite} (that is, every element of $V$ lies in only finitely many members of $\mathcal{F}$). Consider again an interaction $\Phi:\bigcup_{S\in \mathcal{F}}A^S\to [-\infty,\infty)$. In order to encompass models with hard-core constraints, we allow some sets $F$ in $\mathcal{F}$ and certain configurations $\xi_{F} \in A^{F}$ such that $\Phi(\xi_{F}) = -\infty$. Let $\Omega$ be the set of all configurations $\xi$ in $A^{V}$ such that they have \emph{finite local energies}, i.e. for every $F \in \mathcal{F}$, we have $\Phi(\xi_{F}) > -\infty$. It follows that in this case, for every positive integer $s$, we have $$\sum_{F \in \mathcal{F}, F \subseteq T_{d}(s)} \Phi(\xi_{F}) > -\infty.$$  Then a probability measure $\sigma$ on $A^V$ is a \emph{Gibbs state} for $\Phi$ if it is supported on $\Omega$ and for every finite $S\subseteq V$ one has
\begin{equation} \label{DLR}
\sigma_{S} = \int_{A^{V \backslash S}} \mu_{S, \eta} \ \sigma_{V\backslash S}(d \eta),
\end{equation}
where $\sigma_{S}$ is the image of $\sigma$ under the projection map from $A^V$ to $A^S$, and $\mu_{S, \eta}$ is defined as
\begin{equation}\label{DLR conditional}
\mu_{S, \eta}(\xi) \propto \exp\left\{\sum_{F\in \mathcal{F},\, F\subseteq S}\Phi(\xi_F) + \sum_{\substack{F \in \mathcal{F}\\ F \cap S \neq \emptyset, F \cap S^{c} \neq \emptyset}} \Phi((\xi\vee \eta)_F)\right\} \quad \forall \ \xi \in A^S.
\end{equation}
In the last formula, $\xi \vee \eta$ is the element in $A^V$ which equals $\xi$ when restricted to $S$ and equals $\eta$ when restricted to $S^{c}$. The expression in \eqref{DLR conditional} makes sense because we only consider $\eta \in \Omega_{S^{c}} \subseteq A^{S^{c}}$, where $\Omega_{S^{c}}$ is the projection of $\Omega$ onto $A^{S^{c}}$. The quantity $(\xi\vee\eta)_F$ is the restriction of $\xi \vee \eta$ to $F$. Since $\eta \in \Omega_{S^{c}}$, there exists at least one $\xi \in A^{S}$ for which the sum inside the exponential in \eqref{DLR conditional} is finite.

\par Our main result concerns the following setup. We consider the infinite $d$-regular tree $T_{d}$ and a sequence $\{G_{n}\}_{n \in \mathbb{N}}$ of finite graphs locally converging to $T_{d}$ (see Definition \ref{locally tree like}).  Let $\Phi$ be an interaction on $T_d$, and suppose that it has finite range and is invariant under all graph automorphisms of $T_d$.  Then there is a natural way to turn $\Phi$ into an interaction on each $G_n$: one identifies finite subsets of $G_n$ with finite subsets of $T_d$ using the graph structure, and then identifies the corresponding contributions to the interactions.  For instance, the Ising model on $T_{d}$ gives rise to the Ising model on each $G_{n}$ at the same inverse temperature.

This procedure for constructing the Hamiltonians is given precisely in Definition \ref{derived hamiltonian}.  In that definition and the rest of the paper, we also allow more general interactions that are invariant under fewer symmetries of $T_d$. Specifically, we identify $T_d$ with a Cayley graph of a certain free-product group $\Gamma$, and require that $\Phi$ be invariant only under translations, not necessarily under rooted automorphisms.  To transport an interaction with only these fewer symmetries from $T_d$ to the graphs $G_n$, one must endow each $G_n$ with a directed edge-labeling by the generators of $\Gamma$ (although, to be precise, the directions of the edges are relevant only if $d$ is odd).  This is used to specify the correct identification of subsets of $G_n$ with subsets of $T_d$. This is explained carefully in Section \ref{cayley, local weak*}.

In another sense, we also lose some generality as a result of imposing this extra structure. In case $d$ is odd, the group theoretic construction we use does not allow us to consider arbitrary $d$-regular graphs, but only those that are $1$-factorizable: see construction~\ref{perm Z mod 2} in Section~\ref{cayley, local weak*}.

\par In the case of allowing hard-core constraints, we also assume that for every $G_{n} = (V_{n}, E_{n})$ in the sequence, there exists some configuration $\xi \in A^{V_{n}}$ such that its energy is finite according to the interaction $\Phi$. This allows the Gibbs measure on $G_{n}$ associated with $\Phi$ to be well defined.

\par Having transported $\Phi$ to an interaction on each $G_n$, let $\mu_{n}$ be the Gibbs measure thus obtained on $G_{n}$. We shall make the assumption that the sequence $\{\mu_{n}\}$ converges in a local weak$^*$ sense to a Gibbs measure $\mu$ for $\Phi$ on $T_{d}$. In the case of full automorphism-invariance, the local weak$^*$ convergence of measures is taken from \cite{montanari-mossel-sly}.  The definition in our less symmetric setting appears previously in~\cite{austin-sofic}, and is recalled in Definition~\ref{local weak*} below.  We denote this convergence by $\mu_{n} \xrightarrow{lw^{*}} \mu$.

Assuming that $\mu_{n} \xrightarrow{lw^{*}} \mu$, our interest is in finding an expression for the asymptotic behaviour of the specific entropies $H(\mu_{n})/|V_{n}|$ in terms of $\mu$. The candidate that we study here is the \emph{percolative entropy} $H_{perc}(\mu)$.  To define this, let $\phi$ be a root of $T_d$ (later identified with the identity in the group $\Gamma$), and for each $p \in [0,1]$ let $\theta_p$ be the law of a random subset $S$ of $V(T_d)$ which contains each vertex independently with probability $p$. Then we define
\begin{equation}\label{percolative}
H_{perc}(\mu) := \int_{0}^{1} \int_{S \subseteq V(T_{d})} H_{\mu}\left(\sigma_{\phi} | \sigma_{S \backslash \{\phi\}}\right) \theta_{p}(dS) dp.
\end{equation}
In the integrand on the right, we consider a random configuration $\sigma \sim \mu$, and use the conditional Shannon entropy of its value at $\phi$ given its restriction to $S\setminus \{\phi\}$.  This formula is suggested by classical work in ergodic theory, where an analogous quantity is known to equal the Kolmogorov--Sinai entropy of a probability-preserving action of an amenable group~\cite[Theorem 3]{kieffer}.  (It has been compared with the newer notion of sofic entropy for actions of free groups, but we believe it does not agree with sofic entropy in such generality.)

\par Our main result shows that percolative entropy is always an upper bound for the growth rate of the finitary Shannon entropies, and that they are equal in case the model exhibits strong spatial mixing (recalled in Definition \ref{strong spatial mixing} below).

\begin{theorem} \label{main}
Let $\mu_{n}$ denote the Gibbs measure on $G_{n}$ derived from $\Phi$ for $n \in \mathbb{N}$, and assume $\mu_{n} \xrightarrow{lw^{*}} \mu$. Then
\begin{equation}\label{limsup}
\limsup_{n \rightarrow \infty} \frac{1}{|V_n|} H(\mu_{n}) \leq H_{perc}.
\end{equation}
If, moreover, $\Phi$ exhibits strong spatial mixing, then
\begin{equation}\label{limit}
\lim_{n \rightarrow \infty} \frac{1}{|V_n|} H(\mu_{n}) = H_{perc}.
\end{equation}
\end{theorem}

\par Several different definitions of strong spatial mixing already appear in the literature.  Most works define it only for one or another specific model.  In order to use the thermodynamic formalism, we give a definition for general interactions.

Consider again a graph $G = (V, E)$ and a finite set $A$. Let $\Phi$ be an interaction over $V$.  Another collection of functions $\Psi = \{\Psi_v:\ v\in V\}$, where $\Psi_v:A\to \mathbb{R}$, is called a \emph{self-interaction} (or, in some settings, an external field): it may be identified with an interaction with the special property that it is nonzero only for singleton subsets of $V$. Here we allow $\Psi$ to take only finite values. By $\Phi + \Psi$ we mean the interaction defined by
\[
 (\Phi + \Psi)(\xi_{S})= 
  \begin{cases} 
   \Phi({\xi_{v}}) + \Psi(\xi_{v}) & \text{if } S = \{v\}, \text{ for } v \in V, \\
   \Phi(\xi_{S})       & \text{if } |S| > 1,
  \end{cases}
\]
for $S \in \mathcal{F}$, where we abuse notation slightly by writing $\xi_{\{v\}}$ as $\xi_{v}$.  Let $T_d(r)$ denote the closed ball of radius $r$ about the root in the graph $T_d$, and let $T_d(r)^c$ denote its complement in $V(T_d)$.

\begin{defn} \label{strong spatial mixing}
A translation-invariant interaction $\Phi$ on $T_{d}$ exhibits \emph{strong spatial mixing} if the following condition is satisfied:
\begin{equation} \label{strong spatial eq}
\max_{\eta, \tau \in \Omega_{T_d(r)^c}} \left|\mu_{\Phi + \Psi, T_{d}(r), \eta}\left(\sigma_{\phi}=a\right) - \mu_{\Phi + \Psi, T_{d}(r), \tau}\left(\sigma_{\phi}=a\right)\right| \rightarrow 0 \text{ as } r \rightarrow \infty
\end{equation}
for every $a \in A$, uniformly in $\Psi$. Here $\mu_{\Phi + \Psi, T_{d}(r), \eta}$ (similarly $\mu_{\Phi + \Psi, T_{d}(r), \tau}$) is the conditional Gibbs measure corresponding to the interaction $\Phi + \Psi$ with boundary conditions $\eta$ (correspondingly $\tau$), as defined in \eqref{DLR conditional}.
\end{defn}
We consider $\eta$ and $\tau$ only in $\Omega_{T_d(r)^c}$ so that the conditional Gibbs measures in \eqref{strong spatial eq} are well defined.

\par Several well-known models exhibit strong spatial mixing.  These include the Ising and Potts models at sufficiently high temperatures.  Other examples are independent-set and various colouring models, again in their uniqueness regimes.  Some of these examples are discussed in more detail in Section~\ref{sec:egs}.

\section{Cayley graphs and local weak$^*$ convergence in the context of edge-coloured graphs}\label{cayley, local weak*}

Let $G_n = (V_n,E_n)$, $n\in\mathbb{N}$, be a sequence of $d$-regular graphs, and let $U_n$ be the uniform distribution on the vertex set $V_n$.  For $v \in V_n$ and $r \geq 0$, let $B(v,r)$ be the closed ball of radius $r$ about $v$ for the graph metric on $G_n$.

\begin{defn}\label{locally tree like}
We say that $\{G_{n}\}$ \emph{locally converges} to the regular infinite tree $T_{d}$ if for any $t \in \mathbb{N}$,
\begin{equation} \label{locally treelike eq}
\lim_{n \rightarrow \infty} U_{n}\left\{v:\ B(v, t) \cong T_{d}(t)\right\} = 1.
\end{equation}
\end{defn}

Once $n$ is large, this means that most vertices $v \in V_n$ are surrounded by a large-radius ball $B(v,r)$ which is isomorphic to the corresponding ball $T_d(r)$ in the $d$-regular tree.  This isomorphism is the basis for comparing a measure $\mu_n$ on $A^{V_n}$ (or rather its marginal on $A^{B(v,r)}$) to a limiting measure $\mu$ on $A^{T_d}$ (or rather its marginal on $A^{T_d(r)}$).

To make this comparison precise, we must choose a particular isomorphism from $B(v,r)$ to $T_d(r)$.  In case the limiting measure $\mu$ is invariant under rooted automorphisms of the tree, the choice is essentially irrelevant.  However, in this paper we allow a more generous class of models which do not have so much symmetry, and so need to be careful about choosing these neighbourhood-isomorphisms.

In order to do this, we enhance the graphs $G_n$ and the tree $T_d$ with some additional structure that makes the choice of isomorphism canonical.  As a result, our analysis also applies to some models that have less symmetry than $T_d$ itself.

The enhanced structure on our graphs is best described by identifying them with certain Cayley or Schreier graphs of groups, and then colouring each edge according to which generator produced it.  We first describe the identification of $T_d$ with a Cayley graph of a free group or free product.  We then describe the corresponding constructions for finite regular graphs. Background on Cayley graphs and free groups can be found in~\cite{harpe}.

We consider two different ways in which $T_{d}$ may be identified with a Cayley graph.  The first is simpler, but applies only when $d$ is even.  The second applies for any $d$.  
\begin{enumerate}[(1)]
\item \label{free group 1} Suppose $d = 2k$ is even. Consider the free group generated by $k$ generators, which we denote by $\Gamma = \langle s_{1}, s_{2}, \ldots s_{k} \rangle$. Each $s_{i}$ generates an infinite cyclic group. The elements of $\Gamma$ may be written as reduced words in $s_1$, \dots, $s_k$: that is, words of the form $s_{i_{1}}^{j_{1}} s_{i_{2}}^{j_{2}} \ldots s_{i_{m}}^{j_{m}}$ with $j_{1}, \ldots, j_{m}$ nonzero integers and $i_{t} \neq i_{t+1}$ for all $1 \leq t \leq m-1$. Two reduced words are multiplied by concatenating them and then removing cancellations until a new reduced word remains.  The identity element is the empty word, corresponding to $m=0$.

\par The Cayley graph $Cay(\Gamma)$ corresponding to the generators $s_1$, \dots, $s_k$ is a graph with vertex set $\Gamma$. In this paper we regard it as a directed graph with coloured edges in the following way. We connect the directed edge from $w$ to $w s_{i}$ for each $1 \leq i \leq k$ and each $w \in \Gamma$, and we assign the colour $s_{i}$ to this edge.

The underlying undirected graph of $Cay(\Gamma)$ is isomorphic to the $d$-regular tree.  We always use an isomorphism that maps the identity element of $\Gamma$ to the root of the tree, although this does not specify the isomorphism completely.  Given any particular choice of isomorphism, we may carry the directions and colours of the edges of $Cay(\Gamma)$ to $T_d$.  Having done so, each node of $T_d$ has $k$ edges coloured $s_{1}, \ldots s_{k}$ emerging out of it, and $k$ edges coloured $s_{1}, \ldots s_{k}$ entering into it.  This extra structure on $T_d$ then uniquely determines the graph isomorphism that we used.


\item \label{Z mod 2} For any $d \in \mathbb{N}$, let $\Gamma$ be the free product of the groups generated by $s_{1}, \ldots, s_{d}$, where $s_{i}^{2} = e$ for every $1 \leq i \leq d$. Here $e$ denotes the identity element. Then $\langle s_{i} \rangle$ is isomorphic to $\mathbb{Z}/2 \mathbb{Z}$ for each $i$. Elements of this group $\Gamma$ may be written as words of the form $s_{i_{1}} \ldots s_{i_{m}}$ for $m \in \mathbb{N} \cup \{0\}$, with $i_{t} \neq i_{t+1}$ for all $1 \leq t \leq m-1$. The group operation is concatenation followed by removal of cancellations, as in \ref{free group 1}.

\par The corresponding Cayley graph $Cay(\Gamma)$ is now an undirected edge-coloured graph on $\Gamma$. For every $w \in \Gamma$, $w$ is connected by an undirected edge to $w s_{i}$ for every $1 \leq i \leq d$, and that edge is assigned the colour $s_{i}$. There is no value in directing the edges of this Cayley graph, because we have $(ws_i)s_i = ws_i^2 = w$, so each undirected edge would simply be given both directions. The graph $Cay(\Gamma)$ is isomorphic to the infinite $d$-regular tree $T_{d}$. As before, we always map the identity element to the root, and then a particular choice of isomorphism is specified by the image of the colouring on $T_d$.  In this colouring, each vertex of $T_{d}$ is incident upon one edge of each of the colours $s_{1}, \ldots s_{d}$.


\end{enumerate}

Now consider a finite $d$-regular graph $G = (V,E)$.  Corresponding to \ref{free group 1} and \ref{Z mod 2} above, we consider two ways of producing $G$ from a collection of permutations of $V$.

\begin{enumerate}[(i)]
	\item \label{graph case 1} If $d = 2k$, consider $k$ permutations $\gamma^{(1)}, \ldots, \gamma^{(k)}$ of $V$. The set $E$ of edges consists of all pairs $\{u,v\}$ such that $\gamma^{(i)}(u) = v$ for some $i$. Each of these edges $\{u,v\}$ can now be labeled by the $\gamma^{(i)}$ that generated it, and also directed from $u$ to $v$.  It may happen that a single undirected edge receives multiple labels (in case $v = \gamma^{(i)}(u) = \gamma^{(j)}(u)$ for some $i\neq j$) or receives both directions (in case $\left(\gamma^{(i)}\right)^2(u) = u$). However, an arbitrary $d$-regular graph may be generated by $k$ permutations so that neither of these occurs: first, write the edge set as a union of $d$ families of disjoint cycles, and now choose one permutation that generates each of those cycle-families.  We will henceforth consider $d$-regular graphs endowed with the extra structure of these labels and directions.
	
Let $\Gamma$ be as in construction \ref{free group 1} above, and fix an identification of $Cay(\Gamma)$  with $T_d$. For any $v \in V$, the permutations $\gamma^{(1)}, \ldots, \gamma^{(k)}$ specify a canonical graph homomorphism from $T_d$ to $G$ such that $\phi$ is mapped to $v$.  To write it explicitly, we first define a group homomorphism $\gamma: \Gamma\to Sym(V)$ by setting $\gamma(s_i) = \gamma^{(i)}$.  Then the desired graph homomorphism is given by $g \mapsto \gamma(g)(v)$ for $g \in \Gamma$.
		
	\item \label{perm Z mod 2} For arbitrary $d$, we may repeat the construction above, but now use $d$ permutations $\gamma^{(1)}$, \dots, $\gamma^{(d)}$ which have order $2$: that is, $\left(\gamma^{(i)}\right)^2(u) = u$ for every $i$ and $u$. This produces a $d$-regular graph provided $\gamma^{(i)}(u) \neq \gamma^{(j)}(u)$ whenever $i\neq j$. We may then label each edge by the unique $\gamma^{(i)}$ which generated it.
	
	This construction of $G$ is equivalent to writing it as an edge-disjoint union of $d$ matchings. Unfortunately, not every $d$-regular graph may be constructed this way: those that can are called $1$-factorizable.  The Petersen graph is a well-known $3$-regular graph which is not $1$-factorizable.  So at this point our paper loses some generality.  We believe that other constructions could be found to cover this gap, but do not pursue them here.
	
	Given the order-$2$ permutations $\gamma^{(1)}$, \dots, $\gamma^{(d)}$ and a vertex $v \in V$, and letting $\Gamma$ be the free product of copies of $\mathbb{Z}/2\mathbb{Z}$ in construction \ref{Z mod 2} above, we obtain a canonical group homomorphism $\gamma: \Gamma\to Sym(V)$ by setting $\gamma(s_i) := \gamma^{(i)}$. Then a graph homomorphism from $T_d$ to $G$ which sends $\phi$ to $v$ is given by $g\mapsto \gamma(g)(v)$.
\end{enumerate}

\par We now discuss local weak$^*$ convergence of measures in the presence of this extra structure on our graphs.  Compared with \cite{montanari-mossel-sly}, our discussion is less general in one respect but more general in another.  On the one hand, we consider only local convergence to trees, rather than the more general limit graphs allowed in~\cite{montanari-mossel-sly}.  But on the other hand, the structures described above on our graphs and trees are not invariant under all automorphisms of $T_d$.  Nevertheless, no really new ideas are needed.

Consider a finite graph $G = (V,E)$ as described in \ref{graph case 1} or \ref{perm Z mod 2} above, and let $\gamma:\Gamma\to Sym(V)$ be the group homomorphism constructed there.  Let $A$ be any finite set. For any $\vec{x} \in A^V$, any $v \in V$, and any $r \in \mathbb{N}$, the \emph{pull-back name of radius $r$ of $\vec{x}$ at $v$} is defined as
\begin{equation}\label{pull back radius r}
\Pi_{v}^{\gamma, r}(\vec{x}) = \left(x_{\gamma(g)(v)}\right)_{g \in T_d(r)} \in A^{T_d(r)},
\end{equation}
where we identify $T_d(r)$ with the corresponding ball in $\Gamma$ in order to define $\gamma(g)$ for $g \in T_d(r)$. The \emph{pull-back name of $\vec{x}$ at $v$} is
\begin{equation}\label{pull back}
\Pi_{v}^{\gamma}(\vec{x}) = \left(x_{\gamma(g)(v)}\right)_{g \in T_d} \in A^{T_d},
\end{equation}
where we now identify $T_d$ with the whole of $\Gamma$.

Suppose now that $\{G_n\}_n$ is a sequence of $d$-regular graphs which converges locally to $T_d$, and let $\gamma_n: \Gamma\to Sym(V_n)$ be the corresponding sequence of group homomorphisms.  In the language of ergodic theory, such a sequence is an example of a sofic approximation to $\Gamma$: see, for instance,~\cite{austin-sofic} and the references given there.

\begin{defn}\label{local weak*}
Let $\mu_{n}$ be a probability measure on $A^{V_{n}}$ for each $n$, and let $\mu$ be a measure on $A^{T_d}$ which is invariant under the action of $\Gamma$ by left-translation of the coordinates. Let $U_{n}$ be the uniform distribution on $V_n$. Then $\{\mu_{n}\}$ is said to converge to $\mu$ in the \emph{local weak$^*$} sense, if 
\begin{equation} \label{local weak* eq}
\lim_{n \rightarrow \infty} U_{n}\left\{v:\ ||(\Pi_{v}^{\gamma_{n}, r})_\ast \mu_{n} - \mu_{T_d(r)}||_{TV} > \epsilon\right\} = 0 \quad \hbox{for every}\ r > 0,
\end{equation}
where $\mu_{T_d(r)}$ is the marginal of $\mu$ on $T_d(r)$.
\end{defn}

Definition~\ref{local weak*} gives the sense in which we will compare Gibbs measures over the finite graphs $G_n$ with Gibbs measures over $T_d$.  However, to know which Gibbs measures enter that comparison, we must also explain the correspondence between the Hamiltonians that generate them.

Thus, let $\Phi = (\Phi_F)_{F\in\mathcal{F}}$ be an interaction on the tree $T_d$, and assume that it has finite range and is invariant under all translations when $T_d$ is identified with $\Gamma$, one of the groups considered in \ref{free group 1} or \ref{Z mod 2} above.  Let $G = (V,E)$ be a $d$-regular graph endowed with edge colourings and (in case \ref{graph case 1}) edge directions, and let $\gamma:\Gamma\to Sym(V)$ be the resulting group homomorphism.  We next show how to construct an interaction on $G$ which corresponds to $\Phi$. We then apply this procedure to each $G_n$ in a sequence locally converging to $T_d$ in order to obtain the Gibbs measures $\mu_n$.

First, let $\mathcal{F}_0 \subseteq \mathcal{F}$ be a transversal for the translation-action of $\Gamma$ on the family of subsets $\mathcal{F}$. Since $\Phi$ has finite range, the family $\mathcal{F}_0$ is finite, and $\Phi$ is uniquely determined by the subfamily $(\Phi_F)_{F \in \mathcal{F}_0}$ by translation invariance.  The choice of transversal is a necessary part of our procedure for constructing new interactions on finite $d$-regular graphs.

\begin{defn}\label{derived hamiltonian}
We define the Hamiltonian $U_G^{\Phi}$ on $A^V$ as follows:
\begin{equation} \label{pull back hamiltonian}
U_G^{\Phi}(\sigma) = \sum_{v \in V} \sum_{F \in \mathcal{F}_0} \Phi\left(\left(\Pi_{v}^{\gamma}(\sigma)\right)_{F}\right).
\end{equation}
Here $\left(\Pi_{v}^{\gamma}(\sigma)\right)_{F}$ denotes the restriction of $\Pi_{v}^{\gamma}(\sigma) \in A^{V(T_{d})}$ to $F$. The Gibbs measure $\mu_G$ on $A^V$ defined by
\begin{equation} \label{pull back Gibbs}
\mu_G(\sigma) = Z_G^{-1} \exp \left\{U_G^{\Phi}(\sigma)\right\}, \ \text{ where } Z_G = \sum_{\sigma \in A^V} \exp \left\{U_G^{\Phi}(\sigma)\right\},
\end{equation}
is the \emph{Gibbs measure derived from $\Phi$}.
\end{defn}
This Gibbs measure is well defined if at least one $\sigma \in A^{V}$ has finite energy, i.e. $U_G^{\Phi}(\sigma) > -\infty$, and in this case $\mu$ is supported only on those configurations $\sigma$ that have this property. We assume that at least one such $\sigma$ exists in the rest of the paper.

\section{Proof of the main result}

We start with the proof of the upper bound~\eqref{limsup}.

\begin{proof}[Proof of Theorem~\ref{main}, first part]
For any enumeration of $V_n$, say $v_{1}, \ldots, v_{n}$, the chain rule for Shannon entropy gives
\begin{equation} \label{eq 1}
H(\mu_{n}) = \sum_{i=1}^{n} H_{\mu_{n}}\left(\sigma_{v_{i}}|\sigma_{v_{1}}, \ldots \sigma_{v_{i-1}}\right).
\end{equation}

We will use the average of these equalities over a uniform random choice of the enumeration.  To generate such a random ordering, fix one enumeration ${v_1, \dots, v_n}$, and let $X_{1}, \ldots, X_{n}$ be $n$ independent random variables with distribution $\unif[0,1]$. The values of $X_1$, \dots, $X_n$ are almost surely distinct, and in that case the order of those values in $[0,1]$ specifies an ordering of the indices $1$, \dots, $n$.  It is uniformly random among all such orderings.  Using the corresponding re-ordering of $v_1$, \dots, $v_n$, and taking an expectation, \eqref{eq 1} gives
\begin{align} \label{step 1}
\frac{1}{|V_n|} H(\mu_{n}) =& \frac{1}{|V_n|} E_{X_{1}, \ldots, X_{n}}\left[\sum_{i=1}^{n} H_{\mu_{n}}\left(\sigma_{v_{i}}|\{\sigma_{v_{j}}: X_{j} < X_{i}\}\right)\right] \nonumber\\
=& \frac{1}{|V_n|} \sum_{i=1}^{n} E_{X_{1}, \ldots, X_{n}}\left[H_{\mu_{n}}\left(\sigma_{v_{i}}|\{\sigma_{v_{j}}: X_{j} < X_{i}\}\right)\right] \nonumber\\
=& \frac{1}{|V_n|} \sum_{i=1}^{n} E_{X_{i}} \left[E_{X_{1}, \ldots X_{i-1}, X_{i+1}, \ldots X_{n}}\left[H_{\mu_{n}}\left(\sigma_{v_{i}}|\{\sigma_{v_{j}}: X_{j} < X_{i}\}\right)\right]\right].
\end{align}
Beware that the summation from $1$ to $n$ in~\eqref{step 1} is in the fixed order of $v_1$, \dots, $v_n$, not the order specified by the values $X_i$.

Let us write $\theta_{p}$ for the law of a random subset of any given ground set which contains each vertex independently with probability $p$. The particular ground set in question will always be specified separately. When we condition on $X_{i} = p$, the set $\{v_{j} : X_{j} < X_{i}\} = \{v_{j} : X_{j} < p\}$ has the same distribution as $S\setminus v_i$ when $S$ is drawn from $\theta_p$ with ground set $V_n$. Hence \eqref{step 1} leads to:
\begin{align} \label{step 2}
\frac{1}{|V_n|} H(\mu_{n}) =& \frac{1}{|V_n|} \sum_{i=1}^{n} \int_{0}^{1} \int_{S \subseteq V_{n}} H_{\mu_{n}}\left(\sigma_{v_{i}}|\sigma_{S \backslash v_{i}}\right) d\theta_{p} dp \nonumber\\
=& \int_{0}^{1} \int_{S \subseteq V_{n}} E_I\left[H_{\mu_{n}}\left(\sigma_{I}|\sigma_{S \backslash I}\right)\right] d\theta_{p} dp,
\end{align}
where $I$ is a random vertex with distribution $U_{n}$.

\par Now let $r \in \mathbb{N}$ be arbitrary.  By the data-processing inequality, we have
\begin{equation}
H_{\mu_{n}}(\sigma_{I}|\sigma_{S \backslash I}) \leq H_{\mu_{n}}(\sigma_{I}|\sigma_{(S\cap B(I,r)) \backslash I}).
\end{equation}
Substituting into~\eqref{step 2} this gives
\begin{align} \label{step 3}
\frac{1}{|V_n|} H(\mu_{n}) & \leq \int_{0}^{1} \int_{S \subseteq V_{n}} E_I\left[H_{\mu_{n}}\left(\sigma_{I}|\sigma_{(S \cap B(I, r)) \backslash I}\right)\right] d\theta_{p} dp \nonumber\\
&= \int_{0}^{1} E_I\left[ \int_{S \subseteq B(I, r) \backslash I} H_{\mu_{n}}\left(\sigma_{I}|\sigma_{S}\right) d\theta_{p} \right]dp.
\end{align}

\par Let $\widetilde{V}_{n}^{(r)}$ be the subset of vertices $v$ of $G_{n}$ such that $B(v, r) \cong T_{d}(r)$.  As described in Section~\ref{cayley, local weak*}, if this isomorphism exists, it may be chosen uniquely so as to respect the colours and directions of the edges.  If $v \in \widetilde{V}_n^{(r)}$, then we may use that isomorphism to write
\begin{equation}\label{local-ent-int}
\int_{S \subseteq B(v, r) \backslash v} H_{\mu_{n}}(\sigma_v|\sigma_{S}) d\theta_p = \int_{S \subseteq T_d(r) \backslash \phi}H_{(\Pi^{\gamma_n,r}_v)_\ast\mu_{n}}(\sigma_\phi|\sigma_{S}) d\theta_p.
\end{equation}

Since $G_n$ converges locally to $T_d$, we have $U_n\left(\widetilde{V}_n^{(r)}\right) = 1-o(1)$.  We may therefore substitute~\eqref{local-ent-int} into~\eqref{step 3} to obtain
\begin{equation}\label{bound}
\frac{1}{|V_n|} H(\mu_{n}) \leq \int_0^1E_I\left[\int_{S \subseteq T_d(r) \backslash \phi}H_{(\Pi^{\gamma_n,r}_I)_\ast\mu_{n}}(\sigma_\phi|\sigma_{S}) d\theta_p \right]dp + o(1).
\end{equation}

For each set $S\subseteq T_d(r)\setminus \phi$, the expression $H_\nu(\sigma_\phi|\sigma_{S})$ is a continuous function of the measure $\nu$ on $A^{T_d(r)\setminus \phi}$.  It is therefore uniformly continuous, since the space of these measures is compact.  The same therefore holds for the expression
\[\int_{S \subseteq T_d(r) \backslash \phi}H_\nu(\sigma_\phi|\sigma_{S}) d\theta_p,\]
since it is an average of finitely many such functions.  This expression is also always bounded by $\log |A|$ which is the maximum Shannon entropy among all distributions on $A$.

Since $\mu_n$ converges in the local weak$^\ast$ sense to $\mu$, we have
\begin{equation}
U_n\Big\{\big\|(\Pi^{\gamma_n,r}_v)_\ast\mu_n - \mu_{T_d(r)}\big\|_{TV} \leq \delta\Big\} = 1 - o(1)
\end{equation}
as $n\to\infty$ for any $\delta > 0$.  In light of the continuity described above, this implies that
\begin{multline*}
U_n\left\{v:\ \left|\int_{S \subseteq T_d(r) \backslash \phi}H_{(\Pi^{\gamma_n,r}_v)_\ast\mu_n}(\sigma_\phi|\sigma_{S}) d\theta_p - \int_{S \subseteq T_d(r) \backslash \phi}H_{\mu_{T_d(r)}}(\sigma_\phi|\sigma_{S}) d\theta_p\right| \leq \epsilon\right\} \\= 1 - o(1)
\end{multline*}
as $n\to\infty$ for any $\epsilon > 0$.  Substituting into the right-hand side of~\eqref{bound}, this gives
\begin{align*}
\frac{1}{|V_n|} H(\mu_{n}) &\leq \int_0^1\int_{S \subseteq T_d(r) \backslash \phi}H_{\mu_{T_d(r)}}(\sigma_\phi|\sigma_{S}) d\theta_p dp + o(1)\\
&= \int_0^1\int_{S \subseteq V(T_d)}H_\mu(\sigma_\phi|\sigma_{(S\cap T_d(r))\backslash \phi}) d\theta_p dp + o(1).
\end{align*}

For each fixed $r$ we may let $n\to\infty$ above to obtain
\[\limsup_{n\to\infty}\frac{1}{|V_n|}H(\mu_n) \leq \int_0^1\int_{S \subseteq V(T_d)}H_\mu(\sigma_\phi|\sigma_{(S\cap T_d(r))\backslash \phi}) d\theta_p dp.\]
Finally, letting $r \to \infty$, the integrand $H_\mu(\sigma_\phi|\sigma_{(S\cap T_d(r))\backslash \phi})$ is always bounded by $\log |A|$ and decreases pointwise to $H_\mu(\sigma_\phi|\sigma_{S\backslash \phi})$ as a function of $S$, so the above estimate turns into
\begin{equation*}
\limsup_{n\to\infty}\frac{1}{|V_n|}H(\mu_n) \leq \int_0^1\int_{S \subseteq V(T_d)}H_\mu(\sigma_\phi|\sigma_{S\backslash \phi}) d\theta_p dp = H_{perc}(\mu).
\end{equation*}
\end{proof}

The proof of~\eqref{limit} is more delicate.  We begin with some more notation and a couple of preparatory lemmas.

First, let $\Omega_{n}$ denote the set of all configurations $\xi$ over $G_{n}$ such that $U_{n}^{\Phi}(\xi) > -\infty$, and for any $W \subseteq V_n$ let $\Omega_{n,W}$ be set of all $\xi \in A^W$ which have at least one extension in $\Omega_n$. Also, for any $r\in \mathbb{N}$, let
\[\widetilde{V}_{n}^{(r)} = \big\{v \in V_{n}: B(v, r) \cong T_{d}(r)\big\}\]
The local convergence of $\{G_n\}$ to $T_d$ implies that $U_n\left(\widetilde{V}_{n}^{(r)}\right) = 1 - o(1)$ as $n\to\infty$ for every $r$.

\begin{lemma}\label{intermediate lemma 1}
Let $r,\l \in \mathbb{N}$. If $v \in \widetilde{V}_{n}^{(r+l)}$ and $\eta \in \Omega_{n,B(v,r)^c}$, then there exists $\xi \in \Omega_{T_d(r)^c}$ such that
\begin{equation}\label{xi sigma}
\xi_{T_d(r)^c\cap T_d(r+l)} = \eta_{B(v, r)^c\cap B(v,r+l)}
\end{equation}
up to the identification of the index sets given by the isomorphism $B(v, r+l) \cong T_{d}(r+l)$.
\end{lemma}

\begin{proof}
Let $\sigma \in \Omega_n$ be any extension of $\eta$, and let $\xi = \big(\Pi_{v}^{\gamma_{n}}(\sigma)\big)_{T_d(r)^c}$, as defined in \eqref{pull back}. Then~\eqref{xi sigma} follows at once from the isomorphism of $B(v, r+l)$ and $T_{d}(r+l)$.

To complete the proof we show that $\Pi_{v}^{\gamma_{n}}(\sigma) \in \Omega$, since this implies $\xi \in \Omega_{T_d(r)^c}$. Consider any $F \in \mathcal{F}$. Then we can write $F$ as $F_{0}g$ for some $g \in \Gamma$ and $F_{0} \in \mathcal{F}_{0}$, where $\mathcal{F}_{0}$ is the transversal defined before Definition \ref{derived hamiltonian}. Then 
\begin{align}
\left(\Pi_{v}^{\gamma_n}(\sigma)\right)_{F} 
&= \left(\sigma_{\gamma_{n}(h)(v)}\right)_{h \in \mathcal{F}} \nonumber\\
&= \left(\sigma_{\gamma_{n}(h g) (v)}\right)_{h \in \mathcal{F}_{0}} \nonumber\\
&= \left(\sigma_{\gamma_{n}(h) \gamma_{n}(g) (v)}\right)_{h \in \mathcal{F}_{0}} \nonumber\\
&= \left(\Pi_{\gamma_{n}(g)(v)}^{\gamma_{n}}(\sigma)\right)_{F_{0}}, \nonumber
\end{align}
Since $\sigma \in \Omega_{n}$, it follows that $\Phi\Big(\big(\Pi_{v}^{\gamma_n}(\sigma)\big)_{F} \Big) > -\infty$.
\end{proof}

\begin{lemma}\label{intermediate lemma}
If $\Phi$ exhibits strong spatial mixing, then for any $\epsilon > 0$ there is an $r \in \mathbb{N}$ such that
\begin{equation} \label{int lem eq}
U_n\left\{v:\ \max_{S\subseteq V_n\setminus v,\,\tau \in A^{V_n},\,a \in A}\left|\mu_{n}\left(\sigma_{v} = a|\tau_{S}\right) - \mu_{n}\left(\sigma_{v} = a|\tau_{S \cap B(v, r)}\right)\right| \leq \epsilon\right\} \to 1
\end{equation}
as $n\to\infty$.
\end{lemma}

\begin{proof}
There are two parts to the proof.

\vspace{7pt}

\emph{Part 1.}\quad First, we show that if $\Phi$ exhibits strong spatial mixing, then $\mu_{n}$ satisfies another related condition.  To formulate it, suppose that $\Psi = \{\Psi_{v}: v \in V_{n}\}$ is any given set of self-interactions on $G_{n}$, and write $\mu_{n, \Psi}$ for the Gibbs measure
\begin{equation} \label{nth Gibbs measure with self interactions}
	\mu_{n, \Psi}(\sigma) = Z_{n, \Psi}^{-1} \exp \left\{U_{n}^{\Phi}(\sigma) + \sum_{v \in V_{n}} \Psi_{v}(\sigma_{v})\right\}.
\end{equation}
Observe that $\mu_{n}^{\Phi}$ is supported on $\Omega_{n}$. We now prove the following:

\begin{quote}
\emph{For any $\epsilon > 0$ there exist $r \in \mathbb{N}$ and a sequence of subsets $W_n \subseteq V_n$ such that
\begin{itemize}
\item[(i)] $U_n(W_n) = 1 - o(1)$ and
\item[(ii)] every $v \in W_n$ satisfies
\begin{equation} \label{ssm desired}
\max_{\eta, \tau \in \Omega_{n},\, a\in A} \left|\mu_{n, \Psi}\left(\sigma_{v}=a|\eta_{B(v, r)^{c}}\right) - \mu_{n, \Psi}\left(\sigma_{v}=a|\tau_{B(v, r)^{c}}\right)\right| < \epsilon
\end{equation} 
for any set of self-interactions $\Psi$ on $G_n$.
\end{itemize}}
\end{quote}

Note the important feature that $r$ and $W_n$ are independent of the set of self-interactions $\Psi$.

\par Since we assume that $\Phi:\bigcup_{F\in\mathcal{F}}A^F \to [-\infty, \infty)$ has finite range, there is an $l \in \mathbb{N}$ such that every $F \in \mathcal{F}$ satisfies
\begin{equation} \label{contained in slightly bigger ball}
F \cap T_{d}(r) \neq \emptyset \implies F \subseteq T_{d}(r+l).
\end{equation}

\par Recall the formula~\eqref{pull back hamiltonian} for $U_n^\Phi(\sigma)$, and consider any $v \in \widetilde{V}_{n}^{(r+l)}$.  From \eqref{contained in slightly bigger ball} it follows that if a summand $\Phi\Big(\big(\Pi^{\gamma_n}_v(\sigma)\big)\big|_F\Big)$ in \eqref{pull back hamiltonian} involves any coordinates of the restriction $\sigma_{B(v,r)}$, then
\[\left\{\gamma_{n}(g)(v): g \in F\right\}  \subseteq B(v,r+l),\]
and therefore that summand depends only on the slightly larger restriction $\sigma_{B(v,r+l)}$.  From this and the definitions of $\mu_{n}$ and $\mu_{n, \Psi}$, it follows that the conditional measure $\mu_{n, \Psi}\left(\cdot|\eta_{B(v, r)^{c}}\right)$ is really a function of only $\eta_{B(v, r)^{c} \cap B(v, r+l)}$.

Now let $\Theta = \{\Theta_w:\ w\in V(T_d)\}$ be any set of self-interactions on $T_d$ chosen so that the isomorphism $B(v,r+l) \cong T_d(r+l)$ identifies $\{\Psi_u:\ u \in B(v,r+l)\}$ with $\{\Theta_w:\ w\in T_d(r+l)\}$.  Then the same reasoning as above applies to the conditional measures $\mu_{\Phi + \Theta,T_d(r),\zeta}$ over the tree itself.  For any $\eta \in \Omega_{n,B(v,r)^c}$, let $\xi \in \Omega_{T_d(r)^c}$ be provided by Lemma~\ref{intermediate lemma 1}.  Then we obtain
\begin{align} \label{reducing to tree case}
\mu_{n, \Psi}\left(\sigma_{v}=a|\eta_{B(v, r)^{c}}\right) &= \mu_{n, \Psi}\left(\sigma_{v}=a|\eta_{B(v, r)^{c} \cap B(v, r+l)}\right) \nonumber\\
&= \mu_{\Phi + \Psi,T_d(r),\xi}(\sigma_{\phi}=a).
\end{align}

\par By \eqref{reducing to tree case} and the strong spatial mixing of $\Phi$, it follows that if $r$ is sufficiently large then for all $v \in \widetilde{V}_{n}^{(r+l)}$ we have
\begin{align}
& \max_{\eta,\tau \in \Omega_n, a \in A}\left|\mu_{n, \Psi}\left(\sigma_{v} = a|\eta_{B(v, r)^{c}}\right) - \mu_{n, \Psi}\left(\sigma_{v} = a|\tau_{B(v, r)^{c}}\right)\right| \nonumber\\
=& \max_{\xi,\zeta \in \Omega_{T_{d}(r)^{c}}, a \in A}\left|\mu_{\Phi + \Psi,T_d(r),\xi}(\sigma_{\phi} = a) - \mu_{\Phi + \Psi,T_d(r),\zeta}(\sigma_{\phi} = a)\right| \leq \epsilon. \nonumber
\end{align}
Crucially, this $r$ does not depend on $n$, since it is derived from Definition \ref{strong spatial mixing}. Letting $W_n = \widetilde{V}_n^{(r+l)}$ for this choice of $r$ completes the proof of the desired condition.

\vspace{7pt}

\emph{Part 2.}\quad Fix $\epsilon > 0$, and let $r \in \mathbb{N}$ and subsets $W_n \subseteq V_n$ be given by the condition proved in Part 1.  We now show that, with this choice of $r$, all vertices in $W_n$ also satisfy the estimate inside \eqref{int lem eq}.

\par So now suppose that $v \in V_n$, that $S\subseteq V_n\setminus \{v\}$, and that $\eta \in A^{V_n}$.  Let us abbreviate $B = B(v,r)$, and let $S_1 = S\cap B$.  By the tower property of conditional probability, we may express $\mu_n(\sigma_v=a|\eta_S)$ by first conditioning on all spins in $S \cup B^c$, and then integrating over the values of those not in $S$.  This gives
\begin{equation} \label{eq 1 int lem}
\mu_{n}\left(\sigma_{v}=a|\eta_S\right) = \int_{A^{B^{c}}} \mu_n(\sigma_v=a|\eta_{S_1},\tau_{B^c}) \ d\mu_{n}\left(\tau_{B^{c}}|\eta_S\right),
\end{equation} 
and similarly
\begin{equation} \label{eq 2 int lem}
\mu_{n}\left(\sigma_{v}=a|\eta_{S_1}\right) = \int_{A^{B^{c}}} \mu_n(\sigma_v=a|\eta_{S_1},\tau'_{B^c}) \ d\mu_{n}\left(\tau'_{B^{c}}|\eta_{S_1}\right).
\end{equation}
Observe that the first of these integrals is concentrated on $\Omega_{n,B^{c}}$, while the second may give positive mass to a larger set.

Now let $c > 0$ be a parameter, and consider the set of self interactions $\{\Psi_{u}: u \in V_n\}$ given by
\[\Psi_u(a) = \left\{\begin{array}{ll}c &\quad \hbox{if}\ u \in S_1\ \hbox{and}\ a = \sigma_u\\ 0&\quad \hbox{otherwise.}\end{array}\right.\]
(In particular, $\Psi_u \equiv 0$ for $u \in V_n\setminus S_1$.) Then for any $\tau_{B^c} \in A^{B^c}$ we obtain
\[\mu_{n}\left(\sigma_{v}=a|\eta_{S_1},\tau_{B^c}\right) = \lim_{c\to \infty} \mu_{n, \Psi}\left(\sigma_{v}=a|\tau_{B^c}\right).\]
In the case of allowing hard-core constraints, this is true only if $\tau_{B^{c}} \in \Omega_{n,B^{c}}$. But that does not create a problem, since in \eqref{eq 1 int lem} we integrate with respect to $d\mu_{n}\left(\tau_{B^{c}}|\eta_{S}\right)$, which is concentrated on $\Omega_{n,B^{c}}$. A similar argument applies to \eqref{eq 2 int lem}.

\par Since~\eqref{ssm desired} holds uniformly for all sets of self-interactions, we deduce from this that
\begin{equation} \label{favouring magnetic fields}
	|\mu_n(\sigma_v =a|\eta_{S_1},\tau_{B^c}) - \mu_n(\sigma_v=a|\eta_{S_1},\tau'_{B^c})| \leq \epsilon
\end{equation}
for all $v \in W_n$ and $\tau_{B^c},\tau'_{B^c} \in \Omega_{n,B^c}$.

Combining \eqref{favouring magnetic fields} with \eqref{eq 1 int lem} and \eqref{eq 2 int lem}, we obtain
\begin{align*}
&|\mu_n(\sigma_v=a|\eta_S) - \mu_n(\sigma_v=a|\eta_{S_1})|\\
&= \left|\int \mu_n(\sigma_v=a|\eta_{S_1},\tau_{B^c}) \ d\mu_{n}\left(\tau_{B^{c}}|\eta_{S}\right) - \int \mu_n(\sigma_v=a|\eta_{S_1},\tau'_{B^c}) \ d\mu_{n}\left(\tau'_{B^{c}}|\eta_{S_1}\right)\right|\\
&\leq \iint |\mu_n(\sigma_v=a|\eta_{S_1},\tau_{B^c}) - \mu_n(\sigma_v=a|\eta_{S_1},\tau'_{B^c})|\ d\mu_{n}\left(\tau_{B^{c}}|\eta_{S}\right)\ d\mu_{n}\left(\tau'_{B^{c}}|\eta_{S_1}\right)\\
&\leq \epsilon.
\end{align*}

Since $\epsilon > 0$ was arbitrary, this completes the proof.
\end{proof}

We can now prove that~\eqref{limit} holds in case $\Phi$ exhibits strong spatial mixing.

\begin{proof}[Proof of Theorem~\ref{main}, second part]
From the first part of the proof, we see we need only establish equality in \eqref{step 3}.  To do this, we show that, for any $\epsilon > 0$, there is an $r \in \mathbb{N}$ such that
\begin{multline*}
\bigg|\int_{0}^{1} \int_{S \subseteq V_{n}} E_I\left[H_{\mu_{n}}\left(\sigma_{I}|\sigma_{S \backslash I}\right)\right] d\theta_{p} dp \\ - \int_{0}^{1} \int_{S \subseteq V_{n}} E_I\left[H_{\mu_{n}}\left(\sigma_{I}|\sigma_{S \cap B(I, r) \backslash I}\right)\right] d\theta_{p} dp\bigg| \leq \epsilon
\end{multline*}
for all $n$. Now, for every fixed $v \in V_{n}$ and $\eta \in A^{V_n}$, the conditional entropy \[H_{\mu_{n}}\left(\sigma_{v}|\sigma_{S \backslash v} = \eta_{S\setminus v}\right)\]
is a continuous and bounded function of the conditional distribution
\[\mu_{n}(\sigma_{v} = \cdot\, |\sigma_{S \backslash v} = \eta_{S\setminus v}),\]
and similarly
\[H_{\mu_{n}}\left(\sigma_{v}|\sigma_{S \cap B(v, r) \backslash v} = \eta_{S\cap B(v,r)\setminus v}\right)\]
is a continuous and bounded function of \[\mu_{n}(\sigma_{v}= \cdot\,|\sigma_{S \cap B(v, r) \backslash v} = \eta_{S\cap B(V,r)\setminus v}).\]
Using this fact and Lemma \ref{intermediate lemma}, we can find a large $R$ such that for all $r \geq R$ we have
$$\left|H_{\mu_{n}}\left(\sigma_{v}|\sigma_{S \backslash v}\right) - H_{\mu_{n}}\left(\sigma_{v}|\sigma_{S \cap B(v, r) \backslash v}\right)\right| \leq \epsilon \text{ for all } S \subseteq V_{n}$$
whenever $v$ is in the high-probability set of that lemma.  Averaging over $v$, this gives
$$\max_{S\subseteq V_n}\left|\frac{1}{|V_n|}\sum_{v \in V_{n}} H_{\mu_{n}}\left(\sigma_{v}|\sigma_{S \backslash v}\right) - \frac{1}{|V_n|}\sum_{v \in V_{n}} H_{\mu_{n}}\left(\sigma_{v}|\sigma_{S \cap B(v, r) \backslash v}\right)\right| \leq \epsilon + o(1).$$
Finally, repeating the calculation that led to~\eqref{step 2}, the last estimate turns into
$$\left|\int_{S \subseteq V_{n}} E_I\left[H_{\mu_{n}}\left(\sigma_{I}|\sigma_{S \backslash I}\right)\right] d\theta_{p} - \int_{S \subseteq V_{n}} E_I\left[H_{\mu_{n}}\left(\sigma_{I}|\sigma_{S \cap B(I, r) \backslash I}\right)\right] d\theta_{p}\right| \leq \epsilon + o(1)$$
for any $r \geq R$. This completes the proof.
\end{proof}

\section{Examples}\label{sec:egs}
	We discuss here a few examples of models that exhibit strong spatial mixing.
	
\subsection{Ising model}\label{Ising}
The Ising model on the tree $T_{d}$ has $A = \{1, -1\}$ and $\mathcal{F} = E(T_{d})$, the edge-set of $T_{d}$. Let $\beta\Phi$ be the Ising interaction, where $\beta$ is the inverse temperature. This model exhibits strong spatial mixing for all $\beta$ below the uniqueness threshold. The uniqueness threshold of the Ising model can be found in part (a) of Theorem~12.31 of \cite{georgii}. The strong spatial mixing behaviour in this regime can be deduced from Lemma 4.1 of \cite{glauber}, which states the following. 

\begin{lemma}\label{lem:Ising}
Let $T$ be a finite tree, and $v \neq w$ vertices in $V(T)$. Let $\{\beta_{e}: e \in E(T)\}$ be ferromagnetic interactions on $T$, and let $\{\Psi_u: u \in V(T)\}$ be finite-valued external fields (self-interactions) at the vertices of $T$. Let $\mu_{+, \Psi}$ (respectively $\mu_{-, \Psi}$) denote the Gibbs measure with these interactions and external fields $\Psi$, conditioned on $\sigma_{v} = 1$ (respectively $\sigma_{v} = -1$). Then, for fixed interactions $\beta_e$ and $a \in A$, the difference
\[\mu_{+, \Psi}(\sigma_{w} = 1) - \mu_{-, \Psi}(\sigma_{w} = 1)\]
achieves its maximum when $\Psi \equiv 0$.
\end{lemma}

\par Now consider again the model on $T_d$. Since interactions are only defined on edges, each conditional probability $\mu_{\Phi + \Psi, T_{d}(r), \eta}\left(\sigma_{\phi} = 1\right)$ is really a function of $\eta_{T_{d}(r)^{c} \cap T_{d}(r+1)}$, which is a configuration on the finite annulus between $T_{d}(r)$ and $T_{d}(r+1)$. Thus, if $\eta, \tau \in A^{T_{d}(r)^c}$ differ at exactly one $v \in T_{d}(r)^{c} \cap T_{d}(r+1)$, then Lemma~\ref{lem:Ising} gives
\begin{align}
& \left|\mu_{\Phi + \Psi, T_{d}(r), \eta}\left(\sigma_{\phi} = 1\right) - \mu_{\Phi + \Psi, T_{d}(r), \tau}\left(\sigma_{\phi} = 1\right)\right| \nonumber\\
\leq & \left|\mu_{\Phi, T_{d}(r), \eta}\left(\sigma_{\phi} = 1\right) - \mu_{\Phi, T_{d}(r), \tau}\left(\sigma_{\phi} = 1\right)\right|. \nonumber
\end{align}
Since this is a ferromagnetic model, the more positive boundary condition must give a larger probability of spin $+1$ at any other vertex, and therefore the last absolute value is simply equal to
\[\mu_{\Phi, T_{d}(r), \eta}\left(\sigma_{\phi} = 1\right) - \mu_{\Phi, T_{d}(r), \tau}\left(\sigma_{\phi} = 1\right).\]

Now suppose we have two configurations $\eta, \tau \in A^{T_{d}(r)^c}$ such that for every $v \in T_{d}(r)^{c} \cap T_{d}(r+1)$ we have $\eta_{v} \geq \tau_{v}$. Let us define $$\Delta(\eta, \tau) = \left\{v \in T_{d}(r)^{c} \cap T_{d}(r+1): \eta_{v} > \tau_{v}\right\},$$ and let us enumerate the vertices in $\Delta(\sigma, \tau)$ as $v_{1}, \ldots ,v_{N}$. Let $\tau^{(1)}, \ldots \tau^{(N-1)}$ be the configurations defined as follows:
$$\tau^{(1)}(u) = \tau(u) \text{ if and only if } u \neq v_{1},$$ and for all $1 \leq i \leq N-2$,
$$\tau^{(i+1)}(u) = \tau^{(i)}(u) \text{ if and only if } u \neq v_{i+1}.$$ Set $\tau^{(0)} = \tau$ and $\tau^{(N)} = \eta$. Clearly, every consecutive pair $(\tau^{(i)}, \tau^{(i+1)}), 0 \leq i \leq N-1$, differs only at the single vertex $v_{i+1}$. Hence the estimate above gives
\begin{align}
& \left|\mu_{\Phi + \Psi, T_{d}(r), \tau^{(i+1)}}\left(\sigma_{\phi} = 1\right) - \mu_{\Phi + \Psi, T_{d}(r), \tau^{(i)}}\left(\sigma_{\phi} = 1\right)\right| \nonumber\\
=& \mu_{\Phi + \Psi, T_{d}(r), \tau^{(i+1)}}\left(\sigma_{\phi} = 1\right) - \mu_{\Phi + \Psi, T_{d}(r), \tau^{(i)}}\left(\sigma_{\phi} = 1\right) \nonumber\\
\leq & \mu_{\Phi, T_{d}(r), \tau^{(i+1)}}\left(\sigma_{\phi} = 1\right) - \mu_{\Phi, T_{d}(r), \tau^{(i)}}\left(\sigma_{\phi} = 1\right). \nonumber
\end{align}
Summing over all $0 \leq i \leq N-1$, we get:
\begin{align}\label{step 1 Ising}
& \mu_{\Phi + \Psi, T_{d}(r), \eta}\left(\sigma_{\phi} = 1\right) - \mu_{\Phi + \Psi, T_{d}(r), \tau}\left(\sigma_{\phi} = 1\right) \nonumber\\
& \leq \mu_{\Phi, T_{d}(r), \eta}\left(\sigma_{\phi} = 1\right) - \mu_{\Phi, T_{d}(r), \tau}\left(\sigma_{\phi} = 1\right). 
\end{align}

Now consider \emph{any} two configurations $\eta, \tau \in A^{T_{d}(r)^c}$. Form a new configuration $\zeta \in A^{T_{d}(r)^c}$ such that for all $u \in T_{d}(r)^{c} \cap T_{d}(r+1)$, we have $\zeta_{u} = \max\{\eta_{u}, \tau_{u}\}$. Then the estimate above gives
\begin{align}\label{step 2 Ising}
& \mu_{\Phi + \Psi, T_{d}(r), \zeta}\left(\sigma_{\phi} = 1\right) - \mu_{\Phi + \Psi, T_{d}(r), \eta}\left(\sigma_{\phi} = 1\right) \nonumber\\
& \leq \mu_{\Phi, T_{d}(r), \zeta}\left(\sigma_{\phi} = 1\right) - \mu_{\Phi, T_{d}(r), \eta}\left(\sigma_{\phi} = 1\right), 
\end{align}
and
\begin{align}\label{step 3 Ising}
& \mu_{\Phi + \Psi, T_{d}(r), \zeta}\left(\sigma_{\phi} = 1\right) - \mu_{\Phi + \Psi, T_{d}(r), \tau}\left(\sigma_{\phi} = 1\right) \nonumber\\
& \leq \mu_{\Phi, T_{d}(r), \zeta}\left(\sigma_{\phi} = 1\right) - \mu_{\Phi, T_{d}(r), \tau}\left(\sigma_{\phi} = 1\right). 
\end{align}

If $\beta$ lies in the uniqueness regime for this model, then the unique Gibbs measure must be trivial on the tail $\sigma$-algebra (see Theorem 7.7, part (a), of \cite{georgii}), and this implies that the upper bounds in~\eqref{step 2 Ising} and \eqref{step 3 Ising} both tend to $0$ as $r \rightarrow \infty$. Now~\eqref{step 2 Ising} and~\eqref{step 3 Ising} turn this into a uniform rate of convergence for all possible self-interactions, as required for strong spatial mixing.

\subsection{Independent set counting}\label{indep set}
		For a finite graph $G$, this model is concerned with counting independent subsets of vertices. See~\cite{weitz} for more background. An activity parameter $\lambda > 0$ is fixed, and the probability of $I$, for every independent set $I \subset G$, is $\lambda^{|I|}/Z$, where
		\[Z = \sum_{\substack{I \text{ finite independent set}}}\lambda^{|I|}.\]
		To conform with the notation of the present paper, we identify a subset $I \subset G$ with its indicator function, thought of as a configuration $\sigma \in \{0,1\}^V$. If a vertex $v \in V$ appears in the random set $I$, or equivalently satisfies $\sigma_v = 1$, then we say $v$ is \emph{occupied}.

The thermodynamic formalism gives the extension of this model to an infinite graph $G$.  For $\mathcal{F}$ we take the set of all edges and all vertices, i.e.\ $\mathcal{F} = E \cup V$. So, for any $S \in \mathcal{F}$, $S$ can either be a singleton set, i.e.\ a vertex, or a pair $\{u, v\}$ such that $\{u, v\}$ forms an edge of the graph. We define, for $\sigma \in \{0, 1\}^V$ and $S = \{v\}$ a vertex, the interaction
\[
 \Phi(\sigma_{S}) = 
  \begin{cases} 
   \log \lambda & \text{if } \sigma(v) = 1, \\
   0       & \text{if } \sigma(v) = 0;
  \end{cases}
\]
and for $S = \{u, v\}$ an edge of the graph, we define
\[
 \Phi(\sigma_{S}) = 
  \begin{cases} 
   -\infty & \text{if } \sigma(u) = \sigma(v) = 1, \\
   0      & \text{otherwise. }
  \end{cases}
\]
Let $\Lambda \subseteq V$ be any given subset of vertices, and let $\eta \in \{0,1\}^{\Lambda}$ be a configuration over the vertices in $\Lambda$. If $\eta$ is the indicator of an independent subset for the induced graph on $\Lambda$, then we call it \emph{legitimate}. In case $V\setminus \Lambda$ is finite and $v \in V\setminus \Lambda$, let $\mu_{\Phi, \Lambda^c, \eta}\left(\sigma_{v} = 1\right)$ denote the conditional probability that $v$ is occupied, given that the configuration on $\Lambda$ is $\eta$. The following result is taken from Proposition 2.5 of \cite{weitz}.

\begin{prop}
When $\lambda$ is less than $\lambda_{c}(d) = (d-1)^{d-1}/(d-2)^{d}$, the independent set model on $T_d$ with activity parameter $\lambda$ satisfies the following:  for every $v \in V$, co-finite $\Lambda \subseteq V\setminus \{v\}$, and any two legitimate configurations $\eta, \tau$ on $\Lambda$,
		\begin{equation}
			\left|\mu_{\Phi, \Lambda^c, \eta}\left(\sigma_{v} = 1\right) - \mu_{\Phi, \Lambda^c, \tau}\left(\sigma_{v} = 1\right)\right| = o\left(\rho(v, \Delta)\right),
		\end{equation}
where $\Delta \subseteq \Lambda$ is the subset on which $\eta$ and $\tau$ differ, and $\rho(v, \Delta)$ is the distance between $v$ and $\Delta$ in the graph $G$.
\end{prop}

This does not give strong spatial mixing as in our Definition~\ref{strong spatial mixing}: rather, it gives directly the variant of it in~\eqref{favouring magnetic fields} which concerns conditional distributions with two different boundary conditions.  From this we can obtain Theorem~\ref{main} for this model directly.  It should also be straightforward to recover strong spatial mixing for arbitrary finite self-interactions from this proposition.

\subsection{Colouring of graphs}\label{q colouring}
The $q$-colouring model considers proper colourings of the vertices of a graph $G = (V,E)$ by $q$ colours: that is, colourings in which no adjacent vertices have the same colour.  It is discussed in detail in \cite{ge2011}. This is another model with hard-core constraints: the interaction $\Phi$ is defined on edges of the graph, and any configuration in which two adjacent vertices have  the same colour gets weight $-\infty$. As a result, for any co-finite subset $\Lambda$ of $V$ and any proper $q$-colouring $\eta$ of the induced graph on $\Lambda$, the conditional distribution $\mu_{\Phi, \Lambda^c, \eta}$ is the uniform distribution on all proper $q$-colourings of $G$ which extend $\eta$ (unless there are no such extensions, in which case $\eta \not\in \Omega_{\Lambda}$ and this conditional distribution is not defined). The definition of strong spatial mixing in \cite{ge2011} is inspired by that in \cite{weitz}. The main result of \cite{ge2011} gives the following.

\begin{prop}
For $q \geq 1 + \lceil c (d-1)\rceil$, where $c \approx 1.764$ is the root of $c = \exp(1/c)$, the infinite volume Gibbs measure $\mu$ on $q$-colourings of $T_{d}$ is unique and exhibits the following property: for every co-finite $\Lambda \subseteq V$ and $v \in V\setminus \Lambda$, every colour $i$, and every pair of colourings $\eta, \tau \in \Omega_\Lambda$, we have
$$\left|\mu_{\Phi, \Lambda^c, \eta}\left(\sigma_{v} = i\right) - \mu_{\Phi, \Lambda^c, \tau}\left(\sigma_{v} = i\right)\right| \leq C \exp(-a \rho(v, \Delta))$$ for some positive constants $C$ and $a$, where $\Delta = \{u \in \Lambda: \eta_u \neq \tau_u\}$.
\end{prop}

As for the independent set model, in our setting this leads directly to the estimate~\eqref{favouring magnetic fields}, and hence to Theorem~\ref{main}.

\subsection{Potts model}\label{Potts}
Let us consider the Potts model with $A = \{0, 1, \ldots q\}$ and $\mathcal{F} = E(T_{d})$, the edge set of $T_{d}$. Let $V:= V(T_d)$. The interaction is $\Phi(\sigma_{e}) = \beta \delta_{\sigma_{u},\sigma_{v}}$ where the end points of edge $e$ are $u$ and $v$. For this model, the strong spatial mixing behaviour can be shown, for sufficiently small $\beta$, using the same estimates as for the Dobrushin uniqueness theorem. (We thank Elchanan Mossel for pointing this out to us.)

We follow the treatment of Dobrushin's theorem in \cite{simon}. Let us name the nodes of $V$ as $\phi = v_{0}, v_{1}, \ldots$, in a breadth-first manner starting at the root, where $V(T_{d}(r)) = \{v_{0}, v_{1}, \ldots v_{N(r)}\}$. 

\par Fix an arbitrary set of self-interactions $\Psi = \{\Psi_{v}: v \in V\}$. In the following exposition, for any finite $S \subseteq V$ and $\eta \in A^{V \setminus S}$, we abbreviate the measure $\mu_{\Phi+\Psi, S, \eta}$ to $\mu_{S, \eta}$ for convenience of notation.

For any $i, j \geq 0$, define
$$\rho_{i,j} = \frac{1}{2}\sup\left\{||\mu_{v_{j}, \eta} - \mu_{v_{j}, \gamma}||: \eta, \gamma \in A^{V\setminus \{v_j\}},\ \eta_{v_{k}} = \gamma_{v_{k}} \text{ for all } k \neq i\right\}.$$
Simple estimates show that for $\beta$ sufficiently small we have
\begin{equation}
\alpha := \sup_{j \geq 0} \sum_{i \neq j} \rho_{i,j} < 1
\end{equation}
uniformly over all choices of self-interaction $\Psi$. Also, for any $g \in C\left(A^V\right)$, let $\delta_{v_{i}}(g) = \sup\left\{|g(\sigma) - g(\gamma)|: \sigma_{v_{k}} = \gamma_{v_{k}} \text{ for all } k \neq i\right\}$, and let 
\begin{equation*}
\Delta(g) = \sum_{i=0}^{\infty} \delta_{v_{i}}(g).
\end{equation*}

\par Next, for any $j \geq 0$, any $\sigma \in A^V$ and any $\omega \in A$, we define
\[
 (\omega|_{v_{j}} \sigma)_{v_{k}} 
  \begin{cases} 
   = \omega & \text{ if } j = k; \\
   = \sigma_{v_{k}} & \text{otherwise }.
  \end{cases}
\]
Then $\tau_{v_{j}}: C\left(A^V\right) \rightarrow C\left(A^V\right)$ is defined as
\begin{equation}
\tau_{v_{j}}(g)(\sigma) = \int_{A} g\left(\omega|_{v_{j}} \sigma\right) \ d\mu_{v_{j}, \sigma_{V \setminus \{v_{j}\}}}(\omega).
\end{equation}

For every $1 \leq i \leq N(r)$ and $\eta \in A^{T_d(r)^c}$, the transformation $\tau_{v_{i}}$ preserves the measure $\mu_{T_{d}(r), \eta}$. To see this, let $g \in C\left(A^V\right)$, and compute
\begin{align}
\tau_{v_{i}} g(\sigma \vee \eta) &= \int_{A} g\left(\omega|_{v_{i}} (\sigma \vee \eta)\right) \ d\mu_{v_{i}, (\sigma \vee \eta)_{V \setminus \left\{v_{i}\right\}}}(\omega) \nonumber\\
&= \mu_{v_{i}, \left(\sigma \vee \eta\right)_{V \setminus \left\{v_{i}\right\}}}(g) \nonumber\\
&= E\left[g|\mathcal{F}_{V \setminus \left\{v_{i}\right\}}\right](\sigma \vee \eta), \nonumber
\end{align}
where $\mathcal{F}_{S}$ for $S\subseteq V$ is the $\sigma$-field containing information about the spins on the vertices of $S$, and $E$ denotes conditional expectation according to the measure $\mu_{T_{d}(r), \eta}$. By the tower property of conditional expectations, we get:
\begin{multline}
\mu_{T_{d}(r), \eta}\left(\tau_{v_{i}}(g)\right) = E\left[E\left[g|\mathcal{F}_{V \setminus \left\{v_{i}\right\}}\right]|\mathcal{F}_{T_{d}(r)^{c}}\right](\eta) \\=  E\left[g|\mathcal{F}_{T_{d}(r)^{c}}\right](\eta) = \mu_{T_{d}(r), \eta}(g).
\end{multline}
Applying this fact for $i= N(r),N(r)-1,\dots,1$ in turn, we obtain
\begin{equation}\label{for my better understanding}
\mu_{T_{d}(r), \eta}\left(\tau_{v_{0}} \tau_{v_{1}} \ldots \tau_{v_{N(r)}} g\right) = \mu_{T_{d}(r), \eta}(g).
\end{equation} 

Now let us fix any $a \in A$ and consider the particular function $f \in C\left(A^V\right)$ defined by $f(\sigma) = \mathbf{1}_{\left\{\sigma_{v_{0}} = a\right\}}$. For any positive integer $n$, iterating equation~\eqref{for my better understanding} gives
\begin{equation*}
\mu_{T_{d}(r), \eta}\left(\left(\tau_{v_{0}} \ldots \tau_{v_{N(r)}}\right)^{n} f\right) = \mu_{T_{d}(r), \eta}(f)= \mu_{T_{d}(r), \eta}\left(\sigma_{v_{0}} = a\right).
\end{equation*}
Therefore, for any positive integer $n$, we obtain the estimate
\begin{align}\label{iterated Dobrushin}
& \sup_{\eta, \xi \in A^{T_d(r)^c}}\left|\mu_{T_{d}(r), \eta}\left(\sigma_{\phi} = a\right) - \mu_{T_{d}(r), \xi}\left(\sigma_{\phi} = a\right)\right| \nonumber\\
&= \sup_{\eta, \xi \in A^{T_d(r)^c}}\left|\mu_{T_{d}(r), \eta}\left((\tau_{v_{0}} \tau_{v_{1}} \ldots \tau_{v_{N(r)}})^nf\right) - \mu_{T_{d}(r), \xi}\left((\tau_{v_{0}} \tau_{v_{1}} \ldots \tau_{v_{N(r)}})^nf\right)\right| \nonumber\\
&\leq \sup_{\eta, \xi \in A^V}\left|(\tau_{v_{0}} \tau_{v_{1}} \ldots \tau_{v_{N(r)}})^nf(\eta) - (\tau_{v_{0}} \tau_{v_{1}} \ldots \tau_{v_{N(r)}})^nf(\xi)\right| \nonumber\\
&\leq \Delta\left((\tau_{v_{0}} \tau_{v_{1}} \ldots \tau_{v_{N(r)}})^nf\right)
\end{align}


On the other hand, since in the Potts models the interactions are edge-wise, it follows from the definition of $\tau_{v_{j}}$ that $(\tau_{v_{0}} \tau_{v_{1}} \ldots \tau_{v_{N(r)}})^{r'} f$ depends only on $\sigma_{T_{d}(r')}$ for each $r' \le r$. We now combine this fact with~\cite[Equation (V.1.17)]{simon}, obtaining the following inequalities for each $r'\leq r-1$:
\begin{align}
&\Delta\left(\left(\tau_{v_{0}} \tau_{v_{1}} \ldots \tau_{v_{N(r)}}\right)^{r'}f\right) \nonumber\\ &\leq \alpha \sum_{i=0}^{N(r)-1} \delta_{v_{i}}\left(\left(\tau_{v_{0}} \tau_{v_{1}} \ldots \tau_{v_{N(r)}}\right)^{r'-1}f\right) \nonumber\\& \quad + \sum_{N(r)+1}^{\infty}\delta_{v_{i}}\left(\left(\tau_{v_{0}} \tau_{v_{1}} \ldots \tau_{v_{N(r)}}\right)^{r'-1}f\right) \nonumber\\& \quad + \left\{\sum_{i=0}^{N(r)-1} \rho_{i, N(r)} \alpha + \sum_{N(r)+1}^{\infty} \rho_{i,N(r)}\right\} \delta_{v_{N(r)}}\left(\left(\tau_{v_{0}} \tau_{v_{1}} \ldots \tau_{v_{N(r)}}\right)^{r'-1}f\right) \nonumber\\
&\leq \alpha \Delta\left(\left(\tau_{v_{0}} \tau_{v_{1}} \ldots \tau_{v_{N(r)}}\right)^{r'-1}f\right).
\end{align}
Concatenating these inequalities, we conclude that
\[\Delta\left(\left(\tau_{v_{0}} \tau_{v_{1}} \ldots \tau_{v_{N(r)}}\right)^{r-1}f\right) \leq \alpha^{r-1}\]
Combining this with~\eqref{iterated Dobrushin}, we deduce that
\begin{equation*}
\sup_{\eta, \xi \in A^{T_{d}(r)^c}}\left|\mu_{T_{d}(r), \eta}\left(\sigma_{\phi} = a\right) - \mu_{T_{d}(r), \xi}\left(\sigma_{\phi} = a\right)\right| \leq \alpha^{r-1}.
\end{equation*}
Since $\alpha < 1$ and this estimate is uniform in $\Psi$, this establishes the desired strong spatial mixing for the Potts model on $T_{d}$ at sufficiently high temperature.

\bibliography{soficbibfile}

\end{document}